\documentclass[12pt, reqno]{amsart}
\usepackage{amsmath, amsthm, amscd, amsfonts, amssymb, graphicx, color}
\usepackage[bookmarksnumbered, colorlinks, plainpages]{hyperref}
\hypersetup{colorlinks=true,linkcolor=red, anchorcolor=green, citecolor=cyan, urlcolor=red, filecolor=magenta, pdftoolbar=true}

\newtheorem{theorem}{Theorem}[section]
\newtheorem{lemma}[theorem]{Lemma}
\newtheorem{prop}[theorem]{Proposition}

\theoremstyle{definition}
\newtheorem{definition}[theorem]{Definition}

\theoremstyle{remark}
\newtheorem{remark}[theorem]{\textbf{Remark}}
\numberwithin{equation}{section}

\begin{document}

\title[On the Davis-Wielandt shell and the Davis-Wielandt index]{On the Davis-Wielandt shell of an operator and the Davis-Wielandt index of a normed linear space}

\author[Pintu Bhunia, Debmalya Sain and  Kallol Paul]{ Pintu Bhunia, Debmalya Sain and Kallol Paul}

\address[Bhunia] { Department of Mathematics, Jadavpur University, Kolkata 700032, West Bengal, India}
\email{pintubhunia5206@gmail.com}

\address{(Sain) Department of Mathematics, Indian Institute of Science, Bengaluru 560012, Karnataka, India}
\email{saindebmalya@gmail.com}

\address[Paul] {Department of Mathematics, Jadavpur University, Kolkata 700032, West Bengal, India}
\email{kalloldada@gmail.com}

%\thanks will become a 1st page footnote.
\thanks{Pintu Bhunia would like to thank UGC, Govt. of India for the financial support in the form of SRF. Dr. Debmalya Sain feels elated  to acknowledge the joyful presence of his beloved junior Mr. Debabrata Dey in every sphere of his life. Prof. Paul would like to thank RUSA 2.0, Jadavpur University for  partial support. }
\thanks{}

%    General info


%    General info

\date{}
\maketitle

\begin{abstract}
We study the Davis-Wielandt shell and the Davis-Wielandt radius of an operator on a normed linear space $\mathcal{X}$. We show that after a suitable modification, the modified Davis-Wielandt radius defines a norm on  $\mathcal{L}(\mathcal{X})$ which is equivalent to the usual operator norm on $\mathcal{L}(\mathcal{X})$. We introduce the Davis-Wielandt index of a normed linear space and compute its value explicitly in case of some particular polyhedral Banach spaces. We also present a general method to estimate the Davis-Wielandt index of any polyhedral Banach space.
\end{abstract}

\section{Introduction}

\smallskip

\noindent The purpose of this article is to study the Davis-Wielandt shell of an operator between normed linear spaces and some related geometric and analytic properties. Let us first establish the relevant notations and terminologies to be used throughout the article.

\noindent For a complex number $c\in \mathbb{C}$, we write $\mathcal{R}(c)$ and $\mathcal{I}(c)$ for the real part of $c$ and the imaginary part of $c$,  respectively. Let $\mathcal{X}$ be a normed linear space, real or complex and let $\mathcal{X}^*$ denote the dual space of $\mathcal{X}.$  Let $ B_{\mathcal{X}}= \{ x \in \mathcal{X}~:\|x\| \leq 1 \} $ and $  S_{\mathcal{X}}= \{ x \in \mathcal{X}~:\|x\| = 1 \} $ be the unit ball and the unit sphere of $ \mathcal{X}$, respectively. Given a convex subset $K$ of $ \mathcal{X}, $  let $\textit{ext}~K$ denote the set of all extreme points of $ K. $ The Krein-Milman theorem asserts that $\textit{ext}~ K \neq \emptyset$, if $K$ is non-empty compact.  For each $x\in S_{\mathcal{X}}$, let us define $\mathcal{J}(x)=\left\{ f\in \mathcal{X}^*: f(x)=1=\|f\| \right\}.$ We denote by $ \Pi$ the set $\{(x,f): x\in S_{\mathcal{X}}, f\in \mathcal{J}(x)\}$.
Let $\mathcal{L}(\mathcal{X})$ be the normed linear space of all bounded linear operators on $ \mathcal{X} $, endowed with the usual operator norm. 
%For a given $ T \in \mathcal{L}(\mathcal{X}), $ the norm $\|T\|$ of $T$ is given by $$\|T\|= \sup \left \{ \|Tx\|: x\in S_{\mathcal{X}} \right\}.$$ 
\noindent For a complex Hilbert space $\mathcal{H},$ the Davis-Wielandt shell and the Davis-Wielandt radius of an operator $T\in \mathcal{L}(\mathcal{H})$ (see \cite{D,W}) are defined respectively as: 
\begin{eqnarray*}
 DW(T)&= &\left\{ \left(\langle Tx, x \rangle, \|Tx\|^2 \right) : x\in \mathcal{H}, \|x\|=1 \right\},\\
 dw(T)&= &\sup \left\{ \sqrt{|\langle Tx, x \rangle|^2+ \|Tx\|^4} : x\in \mathcal{H}, \|x\|=1 \right \}.
\end{eqnarray*}
The above concepts are closely related to the numerical range of the operator $T$, studied initially by  Bauer \cite{B} and Lumar \cite{L}, in the more general setting of normed linear spaces. 
 Given a $ T \in \mathcal{L}(\mathcal{X}), $ the numerical range and the numerical radius of $ T $ are defined respectively as: 
\begin{eqnarray*}
W(T)&=& \left\{ x^*(Tx) : (x,x^*)\in \Pi \right\},\\
 w(T)&=& \sup \left\{ |x^*(Tx)| : (x,x^*)\in \Pi \right\}.
\end{eqnarray*}
Computation of the numerical index of a normed linear space involves the interaction between the geometry of the space and linear operators acting on it. We refer the readers to \cite{DMPW,MJ,MMP1,MMP2,M,SPBB} and references therein for more information on this topic of interest.
Given a normed linear space $\mathcal{X}$, the numerical index of $\mathcal{X}$ is defined as:
 $$n(\mathcal{X})=\inf\{w(T):T\in S_{\mathcal{L}(\mathcal{X})}\}.$$
Motivated by the study on the numerical index of a normed linear space, we introduce another numerical constant associated with a given space, which is related to the study of the Davis-Wielandt shell of an operator and is christened the Davis-Wielandt index. To this end, we first introduce the Davis-Wielandt set of an operator  at a point of unit norm. For $ T \in \mathcal{L}(\mathcal{X}), $ the Davis-Wielandt set of $T$ at the point $x\in S_{\mathcal{X}}$ is defined as: 
$$DW(T_x)=\left \{\left (x^*(Tx), \|Tx\|^2\right ):x^* \in \mathcal{J}(x) \right \}.$$
Next we recall that the Davis-Wielandt shell and corresponding the Davis-Wielandt radius of an operator $T$ on $\mathcal{X}$. The Davis-Wielandt shell and the Davis-Wielandt radius of $ T\in \mathcal{L}(\mathcal{X}) $ are defined respectively as: 
\begin{eqnarray*}
DW(T)&=& \left\{ \left(x^*(Tx), \|Tx\|^2\right) : (x,x^*)\in \Pi \right\},\\
dw(T)&=&\sup \left\{ \sqrt{| x^*(Tx) |^2+\|Tx\|^4}~:(x,x^*)\in \Pi \right\}.
\end{eqnarray*}
\noindent It is easy to observe that the Davis-Wielandt radius $dw(.)$ does not define a norm on $\mathcal{L}(\mathcal{X})$. However, the following modification allows us to define a norm on $\mathcal{L}(\mathcal{X})$. For $ T \in \mathcal{L}(\mathcal{X}), $ the modified Davis-Wielandt shell and the modified Davis-Wielandt radius of $ T $ are defined respectively as: 
\begin{eqnarray*}
 DW^*(T)&=& \left\{ \left(x^*(Tx), \|Tx\|\right) : (x,x^*)\in \Pi \right\},\\
dw^*(T)&=&\sup \left\{ \sqrt{| x^*(Tx) |^2+\|Tx\|^2}~:(x,x^*)\in \Pi \right\}.
\end{eqnarray*}
\noindent Now we are in a position to introduce the Davis-Wielandt index and the modified Davis-Wielandt index of a normed linear space.
Given a normed linear space $ \mathcal{X} $, the Davis-Wielandt index and the modified Davis-Wielandt index of $ \mathcal{X} $ are defined respectively as: 
\begin{eqnarray*}
\eta_{dw}(\mathcal{X})&=&\inf \left\{ dw(T):T\in S_{\mathcal{L}(\mathcal{X})} \right \},\\
 \eta_{dw^*}(\mathcal{X})&=&\inf \left\{ dw^*(T):T\in S_{\mathcal{L}(\mathcal{X})} \right \}.
\end{eqnarray*}

\smallskip
  
\noindent We are especially interested in computing the (modified) Davis-Wielandt index of polyhedral normed linear spaces. Let us first recall that a finite-dimensional real Banach space is said to be polyhedral if $\textit{ext}~ B_{\mathcal{X}}$ is finite. The following definition, introduced in \cite{SPBB}, is useful in studying the geometry of polyhedral Banach spaces.

\begin{definition}
Let $ \mathcal{X} $ be a polyhedral Banach space. Let $ F $ be a facet of the unit ball $ B_{\mathcal{X}} $ of $  \mathcal{X}. $ A functional $ f \in S_{\mathcal{X}^{*}} $ is said to be a \emph{supporting functional corresponding to the facet $ F $ of the unit ball $ B_{\mathcal{X}} $} if the following two conditions are satisfied:\\
$ (i) $ $ f $ attains norm at some point $ v $ of $ F, $\\
$ (ii) $ $ F =(v + \ker f)\cap S_{\mathcal{X}}. $
\end{definition}
\noindent Now we recall the basic definitions related to the concept of smoothness in a normed linear space $\mathcal{X}$. A normed linear space $\mathcal{X}$ is said to be smooth at a non-zero point $x\in {\mathcal{X}}$ if $\mathcal{J}\left(\frac{x}{\|x\|}\right )$ is singleton. $\mathcal{X}$ is said to be smooth if $\mathcal{X}$ is smooth at each non-zero point $x\in \mathcal{X}$.
\noindent Next we define the following two notations useful to the present study. For $T\in \mathcal{L}(\mathcal{X})$, let $W_T$ and $M_T$ be denote the numerical radius attainment set and the norm attainment set of $T$, defined respectively as: 
\begin{eqnarray*} 
W_T&=& \left\{ x\in S_{\mathcal{X}}:  ~\exists~ x^{*}\in \mathcal{J}(x) ~\mbox{such that}~ |x^{*}(Tx)|=w(T) \right \},\\
M_T &=& \left \{ x\in S_{\mathcal{X}}: \|Tx\|=\|T\| \right \}. 
\end{eqnarray*} 

\smallskip

\noindent This article is demarcated into three section, including the introductory one. In the second section we explore the convexity properties of the Davis-Wielandt shell of an operator, both from the local and the global perspective. We also study the norming properties of the modified Davis-Wielandt radius in $\mathcal{L}(\mathcal{X})$ and its connection with the usual operator norm. The third section is devoted to the investigation of the Davis-Wielandt index of a normed linear space. We estimate the Davis-Wielandt index of polyhedral Banach spaces. We also obtain the exact values of the Davis-Wielandt index in case of some particular polyhedral Banach spaces.

\section{On the Davis-Wielandt shell}
\smallskip

\noindent  
We begin this section with the study of the local convexity properties of the Davis-Wielandt set of a bounded linear operator at a point of unit norm. First we prove the following theorem.

\begin{theorem}\label{DWset}
Let $\mathcal{X}$ be a normed linear space. Let $T\in \mathcal{L}(\mathcal{X})$ and let $x\in S_{\mathcal{X}}$. Then $DW(T_x)$ is convex.  If $\mathcal{X}$ is finite-dimensional, then $DW(T_x)$ is the convex hull of $\Big \{ (x^*(Tx), \|Tx\|^2) : x^*\in \mbox{ext}~~B_{\mathcal{X}^*}, x^*(x)=1 \Big \}.$
\end{theorem}

\begin{proof}
Without loss of generality we assume that $ DW(T_x)$ is not singleton. Let $\left (x^*(Tx), \|Tx\|^2 \right )$, $\left (y^*(Tx), \|Tx\|^2 \right ) \in DW(T_x)$, where $x^*,y^*\in \mathcal{J}({{x}})$.  It is easy to see that $\lambda x^*+(1-\lambda)y^* \in S_{\mathcal{X}^*} $ and $(\lambda x^*+(1-\lambda)y^*)(x)=1$, for all $\lambda \in [0,1].$ Therefore, $\lambda \left (x^*(Tx), \|Tx\|^2 \right )+(1-\lambda)\left (y^*(Tx), \|Tx\|^2 \right )= \Big((\lambda x^*+(1-\lambda)y^*)(Tx), \|Tx\|^2 \Big)\in DW(T_x)$, for all $\lambda \in [0,1].$ So, $ DW(T_x)$ is convex.\\
Let dim $(\mathcal{X}) =n$ and let $x^*\in \mathcal{J}({{x}})$. Then $x^*$ can be written as a convex combination of extreme supporting functionals at $x$. Let $x^*_i$ be the extreme supporting functionals at $x$, for $i=1,2,\ldots,m$. Therefore, there exist scalars $ \lambda_1,\lambda_2,\ldots,\lambda_{m}\geq 0$ with $\sum_{i=1}^m\lambda_i=1$ such that $x^*=\sum_{i=1}^m\lambda_ix_i^*$. So, $x^*(Tx)=\sum_{i=1}^m\lambda_i x^*_i(Tx)$. This proves that  $DW(T_x)$ is the convex hull of $\Big \{ (x^*(Tx), \|Tx\|^2) : x^*\in \mbox{ext}~~B_{\mathcal{X}^*}, x^*(x)=1 \Big \}.$
\end{proof}

\begin{remark}
We would like to remark that if $\mathcal{X}$ is a real normed linear space, then it is easy to observe that  $DW(T_x)$ is the line segment joining 
$\Big (\inf \{x^*(Tx): x^*\in \mbox{ext}~~B_{\mathcal{X}^*}, x^*(x)=1\}, \|Tx\|^2 \Big )$ and  $\Big (\sup \{x^*(Tx): x^*\in \mbox{ext}~~B_{\mathcal{X}^*}, x^*(x)=1\}, \|Tx\|^2 \Big )$.
\end{remark}

\noindent Next we study the convexity properties of the Davis-Wielandt shell of an operator $T$ defined on a normed linear space $\mathcal{X}$. In the case of Hilbert space, this convexity property is well-known (see \cite[Th. 2.2 and Th. 2.3]{LPS}, as given in the following theorem.

\begin{theorem}\label{th-2dim}
Let $\mathcal{H}$ be a Hilbert space and let $T\in \mathcal{L}(\mathcal{H})$. Then the following hold true:\\
$(i)$ If dim $(\mathcal{H}) =2$, then $DW(T)$ is convex if and only if $T$ is normal.\\
$(ii)$ If dim $(\mathcal{H}) \geq 3$, then $DW(T)$ is convex.
\end{theorem}

\smallskip

\noindent However, the above convexity property may not hold true in the more general setting of normed linear spaces. To illustrate this, we first make note of the following proposition.

\begin{prop}\label{prop1}
Let $\mathcal{X}$ be a normed linear space and let $T\in \mathcal{L}(\mathcal{X})$. If $W(T)$ is not convex then $DW(T)$ is not convex.
\end{prop}
\begin{proof}
Since $W(T)$ is not convex, there exist $x^*(Tx), y^*(Ty)\in W(T)$ and $\lambda_0\in (0,1)$ such that $\lambda_0x^*(Tx)+(1-\lambda_0) y^*(Ty)\notin W(T)$, where $x,y\in S_{\mathcal{X}}$ and $x^*\in \mathcal{J}({{x}})$, $y^*\in \mathcal{J}({{y}})$. Therefore, there exists no $z^*\in \mathcal{J}({{z}})$, $z \in S_{\mathcal{X}}$ such that   $\lambda_0x^*(Tx)+(1-\lambda_0) y^*(Ty)=z^*(Tz)$.
Clearly, $\left (x^*(Tx), \|Tx\|^2 \right )$, $\left (y^*(Ty), \|Ty\|^2 \right ) \in DW(T)$, but $\lambda_0\left (x^*(Tx), \|Tx\|^2 \right )+(1-\lambda_0)\left (y^*(Ty), \|Ty\|^2 \right ) \notin DW(T)$. So, $DW(T)$ is not convex. 
\end{proof}

\noindent We now show that $DW(T)$ is not necessarily convex for $T\in \mathcal{L}(\mathcal{X}),$ even if dim $(\mathcal{X}) \geq 3$. This illustrates that Theorem \ref{th-2dim} $(ii)$ cannot be extended to the setting of Banach spaces.

\begin{theorem}\label{notconvex}
Let $\mathcal{X}={{\ell}}^n_p$, $p\neq 1,2,\infty$ and $n\geq 2$. Let $T\in \mathcal{L}(\mathcal{X})$ be such that $$T=\left(\begin{array}{ccc}
    a&b \\
    c&d
	\end{array}\right)\oplus 0_{n-2\times n-2},~~ a,d\in \mathbb{R}\setminus \{0\},~~b,c\in \mathbb{C}\setminus \{0\} $$
	and $(i)$~~$a+d=0, (ii)~~ |b|=|c|, (iii)~~ \mathcal{R}(b) \mathcal{I}(c)+ \mathcal{R}(c) \mathcal{I}(b)=0.$ Then $DW(T)$ is not convex.
\end{theorem}

\begin{proof}
Let $S=\left(\begin{array}{ccc}
    a&b \\
    c&d
	\end{array}\right).$ Mandal et al.  proved in \cite[Th. 2.9]{MBBP} that $W(S)$ is not convex. By similar computations we can show that $W(T)$ is not convex. Therefore, it follows from Proposition \ref{prop1} that $DW(T)$ is not convex.
\end{proof}

\smallskip

\noindent Our next result relates the smoothness of a normed linear space at a point of unit norm with the local convexity properties of the Davis-Wielandt set of operators on the space.

\begin{prop}\label{smothness}
Let $\mathcal{X}$ be a normed linear space and let $x\in S_{\mathcal{X}}.$ Then $x$ is a smooth point of $\mathcal{X}$ if and only if $ DW(T_x)$ is singleton, for every $T\in \mathcal{L}(\mathcal{X})$.
\end{prop}
\begin{proof}
The necessary part follows trivially. Let us prove only the sufficient part. Let $ DW(T_x)$ be singleton, for every $T\in \mathcal{L}(\mathcal{X})$. Suppose on the contrary that $x$ is not a smooth point. Then there exist distinct supporting functionals $x^*, y^*\in \mathcal{J}(x)$. Let $z\in S_{\mathcal{X}}$ be such that $x^*(z)\neq y^*(z).$ Let $T\in \mathcal{L}(\mathcal{X})$ be such that $Tx=z.$ Therefore, $x^*(Tx)\neq y^*(Tx).$ Thus, there is an operator $T\in \mathcal{L}(\mathcal{X})$ such that $ DW(T_x)$ is not singleton, a contradiction. Hence, $x$ is a smooth point.
\end{proof}

\smallskip
 
\noindent In the next result we show that the modified Davis-Wielandt radius $dw^*(.)$ defines a norm on $\mathcal{L}(\mathcal{X}),$ for any normed linear space $\mathcal{X}.$

\begin{theorem}\label{norm}
Let $\mathcal{X}$ be a normed linear space. Then $dw^*(.)$ defines a norm on $\mathcal{L}(\mathcal{X})$. It is equivalent to the usual operator norm $\|.\|$, satisfying the following inequality: $$\|T\|\leq dw^*(T)\leq \sqrt{2}\|T\|,~~\mbox{for every}~~ T\in \mathcal{L}(\mathcal{X}).$$ 
\end{theorem}

\begin{proof}
First we prove that $dw^*(.)$ defines a norm on $\mathcal{L}(\mathcal{X})$. Let $T\in \mathcal{L}(\mathcal{X}).$ Clearly,
 $dw^*(T)=0$ if and only if $T=0$ and $dw^*(cT)=|c|dw^*(T)$, for all scalars $c$. Let $T_1, T_2\in \mathcal{L}(\mathcal{X}).$ Then
\begin{eqnarray*}
&& dw^*(T_1+T_2)\\
 = && \sup \left \{\sqrt{ |x^*((T_1+T_2)x)|^2+\|(T_1+T_2)x\|^2 }: (x,x^*)\in \Pi  \right \}\\
  \leq &&  \sup \left \{\sqrt{ |x^*(T_1x)+x^*(T_2x)|^2+(\|T_1x\|+\|T_2x\|)^2 }: (x,x^*)\in \Pi  \right \}\\
	\leq &&  \sup \Big \{ \sqrt{ |x^*(T_1x)|^2+\|T_1x\|^2 }+\sqrt{|x^*(T_2x)|^2+\|T_2x\|^2}: (x,x^*)\in \Pi  \Big \},\\
		 &&  \hspace{6cm}\mbox{by Minkowski inequality}\\
		\leq &&  \sup \left \{ \sqrt{ |x^*(T_1x)|^2+\|T_1x\|^2 }: (x,x^*)\in \Pi  \right \} \\
	&&  + \sup \left \{ \sqrt{ |x^*(T_2x)|^2+\|T_2x\|^2 }: (x,x^*)\in \Pi  \right \}\\
		=&& dw^*(T_1)+dw^*(T_2).
\end{eqnarray*}
Therefore, $dw^*(T_1+T_2)\leq dw^*(T_1)+dw^*(T_2)$ for every $T_1, T_2\in \mathcal{L}(\mathcal{X}).$ Hence, $dw^*(.)$ defines a norm on $\mathcal{L}(\mathcal{X}).$ Next we prove the other part of the theorem. It follows from the definition of the modified Davis-Wielandt radius that
\begin{eqnarray*}
dw^*(T)&=&\sup \left\{ \sqrt{| x^*(Tx) |^2+\|Tx\|^2}~:(x,x^*)\in \Pi \right\}\\
&\leq& \sqrt{  w^2(T)+ \|T\|^2   }\leq \sqrt{  \|T\|^2 + \|T\|^2 }\leq \sqrt{2} \|T\|.
\end{eqnarray*}
Also,  
\begin{eqnarray*}
dw^*(T)&=&\sup \left\{ \sqrt{| x^*(Tx) |^2+\|Tx\|^2}~:(x,x^*)\in \Pi \right\}\\
&\geq& \max \left \{ w(T), \|T\|   \right \}=\|T\|.
\end{eqnarray*}
This completes the proof.

\end{proof}

\begin{remark}
We would like to remark that the inequality in Theorem \ref{norm} is sharp. For the identity operator $T$ on $\mathcal{X}$, the second inequality becomes an equality. For a skew-symmetric operator $T$ on a real Hilbert space $\mathcal{H}$, the first inequality becomes an equality. 
\end{remark}

\noindent We would like to estimate the Davis-Wielandt radius of an arbitrary operator on a given polyhedral space. First we recall the following lemma from \cite{SPBB}, which is useful in this context.

\begin{lemma} \label{lemma:extreme}
Let $\mathcal{X}$ be a polyhedral Banach space. Then $f \in S_{\mathcal{X}^{*}}$ is an extreme point of $ B_{\mathcal{X}^{*}} $ if and only if $ f $ is a supporting functional corresponding to a facet of $B_{\mathcal{X}}.$  
\end{lemma}

\smallskip

\noindent Now we are in a position to prove the following theorem.

\begin{theorem}\label{prop-est}
Let $\mathcal{X}$ be an $n$-dimensional polyhedral Banach space and let $T\in \mathcal{L}(\mathcal{X}).$ Then
\[dw(T)\leq \max \left\{ \sqrt{| x^*(Tx) |^2+\|T\|^4}~:(x,x^*)\in G \right\}, \] 
where $G=\{(x,x^*): x\in \textit{ext}~B_{\mathcal{X}}, x^*\in \textit{ext}~B_{\mathcal{X}^*}, x^*(x)=1 \}.$ 
\end{theorem}

\begin{proof}
It follows from the definition of the Davis-Wielandt radius of $T$ that
\begin{eqnarray*}
dw(T)&= & \sup \left\{ \sqrt{| x^*(Tx) |^2+\|Tx\|^4}~:(x,x^*)\in \Pi \right\}\\
&\leq & \sup \left\{ \sqrt{| x^*(Tx) |^2+\|T\|^4}~:(x,x^*)\in \Pi \right\}.
\end{eqnarray*}
Let $x\in S_{\mathcal{X}}.$ Let $F$ be a facet of $B_{\mathcal{X}}$ containing $x$. Let $\textit{ext}~B_{\mathcal{X}} \bigcap F =\{x_1,x_2,\ldots,x_m\}$. Then there exist scalars $\lambda_1,\lambda_2,\ldots,\lambda_m \geq 0$ with $\sum_{i=1}^m\lambda_i =1$ such that  $x= \sum_{i=1}^m\lambda_ix_i.$ Let $x^*$ be the supporting functional corresponding to the facet $F$. Then by Lemma \ref{lemma:extreme}, we have $x^*\in \textit{ext}~B_{\mathcal{X}^*}$ and  $x^*\in \mathcal{J}(x_i)$ for all $i=1,2,\ldots,m.$ Now,
\begin{eqnarray*}
|x^*(Tx)|&=&\left |\sum_{i=1}^m\lambda_i x^*(Tx_i) \right |\\
&\leq& \sum_{i=1}^m\lambda_i |x^*(Tx_i)|\\
&\leq& \max_{1\leq i\leq m} \left \{ |x^*(Tx_i)|\right\}.
\end{eqnarray*} 
Thus, $|x^*(Tx)|^2+\|T\|^4\leq \max\left \{ |x^*(Tx_i)|^2: 1\leq i\leq m \right\}+\|T\|^4.$ This implies that 
\begin{eqnarray*}
\sup \left\{ \sqrt{| x^*(Tx) |^2+\|T\|^4}~:(x,x^*)\in \Pi \right\}=\max \left\{ \sqrt{| x^*(Tx) |^2+\|T\|^4}~:(x,x^*)\in G \right\}.
\end{eqnarray*}
This completes the proof. 
\end{proof}

\section{On the Davis-Wielandt index}\label{sec3}

\smallskip 

\noindent We begin this section with the following inequalities for the Davis-Wielandt index and the numerical index of a normed linear space $\mathcal{X}$. 

\begin{prop}\label{prop-bound eta}
Let $\mathcal{X}$ be a normed linear space. Then $$0\leq n(\mathcal{X})\leq 1\leq \eta_{dw}(\mathcal{X})\leq \sqrt{n^2(\mathcal{X})+1} \leq \sqrt{2}.$$ 
\end{prop}

\begin{proof}
The inequalities $0\leq n(\mathcal{X})\leq 1\leq \eta_{dw}(\mathcal{X})$ and $\sqrt{n^2(\mathcal{X})+1} \leq \sqrt{2}$ follows trivially. We only prove that $\eta_{dw}(\mathcal{X})\leq \sqrt{n^2(\mathcal{X})+1}.$ Let $T\in S_{\mathcal{L}(\mathcal{X})}.$ Then it follows from the definition of the Davis-Wielandt radius of $T$ that
\begin{eqnarray*}
dw(T)=\sup \left\{ \sqrt{| x^*(Tx) |^2+\|Tx\|^4}~:(x,x^*)\in \Pi \right\} \leq\sqrt{  w^2(T)+ 1 }.
\end{eqnarray*}
Taking infimum over all $T\in S_{\mathcal{L}(\mathcal{X})}$, we get $\eta_{dw}(\mathcal{X})\leq \sqrt{n^2(\mathcal{X})+1}.$
\end{proof}

\smallskip

\begin{remark}\label{remneq}
$(i)$ For the inequalities $ 0\leq n(\mathcal{X})\leq 1$ to be equalities (in separate cases), we refer the readers to the study of normed spaces with numerical index $0$ \cite{MMP1,MMP2} and normed spaces with numerical index $1$ \cite{M}, respectively.\\

\noindent $(ii)$ It follows from Proposition \ref{prop-bound eta} that if $n(\mathcal{X})=0$ then $\eta_{dw}(\mathcal{X})=1.$ However, $\eta_{dw}(\mathcal{X})=1$ does not always imply $n(\mathcal{X})=0$. As for example, we consider $\mathcal{H}$ to be the $n$-dimensional complex Hilbert space, where $n\geq 2$. Let $T=\left(\begin{array}{cc}
    0&1 \\
    0&0
	\end{array}\right) \bigoplus 0$, be the usual matrix representation of $T\in \mathcal{L}(\mathcal{H})$.  It is easy to see that $\|T\|=dw(T)=1$. Since $ \eta_{dw}(\mathcal{H})\geq 1$, it is now immediate that $\eta_{dw}(\mathcal{H})=1$. On the other hand, it is well known that $n(\mathcal{H})=\frac{1}{2}.$
\end{remark}

\noindent In our next result we prove that the following equivalent conditions for the Davis-Wielandt index of a finite-dimensional Banach space to be $\sqrt{2}$.

\begin{prop}\label{equivalent}
Let $\mathcal{X}$ be a finite-dimensional Banach space. Then the following conditions are equivalent:\\
$(i)$ $\eta_{dw}(\mathcal{X})=\sqrt{2}$.\\
$(ii)$ $n(\mathcal{X})=1.$\\
$(iii)$ $w(T)=\|T\|$ for all $T\in {\mathcal{L}(\mathcal{X})}.$\\
$(iv)$ $|x^*(x)|=1$ for every $x\in \textit{ext}~ B_{\mathcal{X}}$ and for every $x^*\in \textit{ext}~ B_{\mathcal{X}^*}.$
\end{prop}

\begin{proof}
\noindent The equivalence of $(ii)$ and $(iv)$ was proved by McGregor in \cite[Th. 3.1]{M}. The equivalence of $(ii) $ and $(iii) $ follows trivially. We only prove the equivalence of $(i) $ and $(iii) $.\\

\noindent $(i) \implies (iii)$: Let $\eta_{dw}(\mathcal{X})=\sqrt{2}$. Then $dw(T)\geq \sqrt{2}$ for all $T\in S_{\mathcal{L}(\mathcal{X})}.$ Therefore, $\sqrt{2} \leq dw(T)\leq \sqrt{w^2(T)+\|T\|^4} $ $\leq \sqrt{2}$ for all $T\in S_{\mathcal{L}(\mathcal{X})}.$ This implies that $w(T)=1$ when $\|T\|=1$. Hence, $w(T)=\|T\|$ for all $T\in {\mathcal{L}(\mathcal{X})}.$\\

\noindent $(iii) \implies (i)$: Let $w(T)=\|T\|$ for all $T\in {\mathcal{L}(\mathcal{X})}.$
Since $\mathcal{X}$ is finite-dimensional, $W_T\neq \emptyset$ for all $T\in {\mathcal{L}(\mathcal{X})}.$ Therefore, it follows from \cite[Th. 3.12]{SMBP} that $W_T \bigcap M_T \neq \emptyset.$ Let $x_0\in W_T \bigcap M_T.$ So, if $\|T\|=1$ then there exist $x^*_0\in \mathcal{J}(x_0)$ such that $w(T)=|x^*_0(Tx_0)|=1$.  Therefore, $dw(T)\geq \sqrt{|x_0^*(Tx_0)|^2+\|Tx_0\|^4}=\sqrt{2}$.  Thus, $dw(T)\geq \sqrt{2}$ for all $T\in S_{\mathcal{L}(\mathcal{X})}$. Hence, $\eta_{dw}(\mathcal{X})\geq \sqrt{2}.$ Also, it follows from Proposition \ref{prop-bound eta} that $\eta_{dw}(\mathcal{X})\leq \sqrt{2}.$  Therefore, $\eta_{dw}(\mathcal{X})=\sqrt{2}$.

\end{proof}

\begin{remark}
The equivalence of  $(i)$, $(ii)$ and $(iii)$ hold true even if $\mathcal{X}$ is infinite-dimensional. The implication $(iv) \implies (ii)$ may not hold true if $\mathcal{X}$ is infinite-dimensional. As for example, we note from \cite{VMR}  that  $\textit{ext}~B_{c_0(\ell_2)}=\emptyset$ but $n(c_0(\ell_2))<1.$
\end{remark}

\smallskip

\noindent For the inequality $\eta_{dw}(\mathcal{X})\leq \sqrt{n^2(\mathcal{X})+1}$ to be an equality, we note that a sufficient condition is $n(\mathcal{X})\in \{0,1\}$. However, there exist polyhedral Banach spaces $\mathcal{X}$, for which $n(\mathcal{X})\neq 0,1$, but the above inequality turns out to be an equality. Before discussing such examples, we refer the readers to \cite{SPBB}, where explicit computation of the exact numerical index of certain $3$-dimensional Banach spaces was done. Since the computation of the Davis-Wielandt index resembles the computation of the numerical index to be considerable extent, we will not repeat the calculations. However, to show that  $\eta_{dw}(\mathcal{X})= \sqrt{n^2(\mathcal{X})+1}$ for the following spaces, denoted in each case by $\mathcal{X}$, we follow the same strategy:\\

\noindent $(1)$ The value of $n(\mathcal{X})$ is known from \cite{SPBB}. We show that ${dw}(T)\geq \sqrt{n^2(\mathcal{X})+1},$ for each $T\in S_{\mathcal{L}(\mathcal{X})}.$ In particular, this implies that $\eta_{dw}(\mathcal{X})\geq \sqrt{n^2(\mathcal{X})+1}$.\\
$(2)$ We explicitly demonstrate an operator $T_0\in S_{\mathcal{L}(\mathcal{X})}$ for which ${dw}(T_0)= \sqrt{n^2(\mathcal{X})+1}.$  This combined with $(1)$ gives the desired equality.\\ 
\noindent In each case, the operator $T_0$ is chosen such that $w(T_0)=n(\mathcal{X})$, as given in \cite{SPBB}.

\smallskip
We state the following theorems that determine the exact value of  the Davis-Wielandt index of some special $3$-dimensional polyhedral Banach spaces, the proof of which is clear from   the above observations coupled with the results obtained in \cite{SPBB}.

\begin{theorem}\label{theorem:pyramid}

Let $\mathcal{X}$  be a $3$-dimensional polyhedral Banach space such that $B_{\mathcal{X}}$ is a polyhedron obtained by gluing two pyramids at the opposite base faces of a right prism having square base, with vertices $ \pm(1,1,1), \pm(-1,1,1), $   $ \pm(-1,-1,1), $ $ \pm(1,-1,1), \pm(0,0,2)$.  Then $$\eta_{dw}(\mathcal{X})= \frac{\sqrt{5}}{2}.$$
\end{theorem}

\smallskip

\begin{theorem}\label{theorem:prism pyramid}
Let $\mathcal{X}$ be a $3$-dimensional polyhedral Banach space such that $B_{\mathcal{X}}$ is a polyhedron with vertices $(\cos(j-1)\frac{\pi}{n}, \sin(j-1)\frac{\pi}{n},\pm{1}), (0,0,\pm 2)$, $ j \in \{1,2,\ldots,2n\}$, $ n\geq 3$. Then \[\eta_{dw}(\mathcal{X})= \begin{cases}
\sqrt{\sin^2\frac{\pi}{2n}+1}, &  \mbox{when n is odd}  \\
\sqrt{\tan^2\frac{\pi}{2n}+1}, &  \mbox{when n is even}.
\end{cases}
\] 

\end{theorem}	

\smallskip

\begin{theorem}\label{theorem:odd}
Let $\mathcal{X}$ be a $3$-dimensional polyhedral Banach space such that $B_{\mathcal{X}}$ is a right prism whose base is a regular polygon having $2n$ sides.  Then \[\eta_{dw}(\mathcal{X})= \begin{cases}
\sqrt{\sin^2\frac{\pi}{2n}+1}, &  \mbox{when n is odd}  \\
\sqrt{\tan^2\frac{\pi}{2n}+1}, &  \mbox{when n is even}.
\end{cases}
\]  
\end{theorem}

\smallskip

\noindent We observe that Theorem \ref{theorem:odd} does not depend on the height $h$ (say) of the prism.  In particular, taking $ h=0$, we obtain the Davis-Wielandt index of the corresponding two-dimensional polyhedral Banach spaces. This fact is recorded in the following theorem.

\begin{theorem}\label{theorem:odd radius2}
Let $\mathcal{X}$  be a $2$-dimensional polyhedral Banach space such that $B_{\mathcal{X}}$ is a  regular $2n$-gon with vertices $\left (\cos(j-1)\frac{\pi}{n}, \sin(j-1)\frac{\pi}{n}\right)$,  $j \in \{1,2,\ldots,2n\}$. Then \[\eta_{dw}(\mathcal{X})= \begin{cases}
\sqrt{\sin^2\frac{\pi}{2n}+1}, &  \mbox{when n is odd}  \\
\sqrt{\tan^2\frac{\pi}{2n}+1}, &  \mbox{when n is even}.
\end{cases}
\]  
\end{theorem}

\smallskip

\begin{remark} 
Comparing the polyhedral Banach spaces considered in Theorem \ref{theorem:pyramid} and Theorem \ref{theorem:prism pyramid} with $n=3$, we conclude that non-isometric Banach spaces may have the same Davis-Wielandt index. 
\end{remark}

\smallskip

\noindent In our next result we obtain a lower bound of the Davis-Wielandt index of polyhedral Banach spaces.

\begin{theorem}\label{theorem:estimate}
Let $\mathcal{X}$ be an $n$-dimensional polyhedral Banach space such that $\textit{ext}~B_{\mathcal{X}}=\{\pm v_i:i\leq i\leq m\}$. let $ F_{i1},F_{i2},$ $ \ldots,F_{in} $ be any $ n $ facets of $ B_{\mathcal{X}} $ meeting at $ v_i. $ Let $ f_{i1},f_{i2}, \ldots, f_{in} $ be  the $ n $ supporting functionals corresponding to the facets $ F_{i1},F_{i2},\ldots,F_{in} $, respectively. For each $ i=1,2,\ldots,m, $ let \[ \xi_i = \min_{x \in S_{\mathcal{X}}}\left\{ \max_{1 \leq r \leq n} \sqrt{|f_{ir}(x)|^2+1 } \right\}. \] 
Then 
\[\eta_{dw}(\mathcal{X}) \geq \min \{\xi_1,\xi_2,\ldots,\xi_m\}.\] 
\end{theorem}

\begin{proof}
It is easy to observe that $ x \longrightarrow \max_{1 \leq r \leq n}\left\{\sqrt{|f_{ir}(x)|^2+1 } \right \} $, for each $ i=1,2,\ldots,m$,  is a continuous function from $S_{\mathbb{X}}$ to $\mathbb{R}$. Since $ S_{\mathbb{X}} $ is compact, $\xi_i= \min_{x \in S_{\mathcal{X}}}\left\{ \max_{1 \leq r \leq n} \sqrt{|f_{ir}(x)|^2+1 } \right\}$ exists. Let $ T \in S_{\mathcal{L}(\mathcal{X})} $ be arbitrary.  $ T $ must attain its norm at an extreme point of $ B_{\mathcal{X}}.$ Thus, there exists $ v_{j} $ such that $ \|Tv_{j}\|=\| T \|=1 $ for some $j \in \{1,2,\ldots,m\}.$ Therefore, we have
\begin{eqnarray*}
 dw(T) &=& \sup \left \{ \sqrt{|x^*(Tx)|^2+\|Tx\|^4}~:(x,x^*)\in \Pi \right \} \\
&\geq & \max_{1 \leq r \leq n} \left \{\sqrt{\vert f_{jr}(Tv_{j})\vert^2+1} \right \} \\
&\geq&\xi_{j}.
\end{eqnarray*}
This gives the desired inequality.
\end{proof}

\smallskip

\noindent As an illustrative application of the above theorem in estimating the Davis-Wielandt index of polyhedral Banach spaces, we consider a family of Banach spaces considered in \cite{MJ}. For each $\gamma \in (0,1),$ define a norm  $ {\Vert .\Vert}_\gamma  $ on $\mathcal{X}_{\gamma}= \{ (x,y,0) :  x,y\in \mathbb{R} \}$ as  \[{\Vert (x,y,0)\Vert}_\gamma=\max \left\{\vert y\vert, \vert x \vert+(1-\gamma)\vert y\vert \right\} ~~\forall~~(x,y,0)\in \mathbb{R}^3.\] The unit ball $B_{\mathcal{X}_{\gamma}}$ is a hexagon with vertices $\pm(\gamma,1,0),\pm(1,0,0),\pm(\gamma,-1,0)$.
We can estimate the Davis-Wielandt index of the $3$-dimensional polyhedral Banach space $ \mathcal{Y}_{\gamma} $ such that $B_{\mathcal{Y}_{\gamma}}$ is a right prism whose base is a hexagon $B_{\mathcal{X}_{\gamma}}.$ 
Let us remark that the proofs are omitted as they are lengthy and full of detailed calculations. We identify the extreme functionals of the unit ball of the dual space as the functionals corresponding to the facets of the unit ball of the space and then study the action of these functionals on the unit sphere to obtain a lower bound for the Davis-Wielandt index. 

\begin{theorem}\label{theorem:hexagon}
Let $\mathcal{Y}_{\gamma}$ be a $3$-dimensional polyhedral Banach space such that $B_{\mathcal{Y}_{\gamma}}$ is a right prism whose base is a hexagon $B_{\mathcal{X}_{\gamma}}$ defined as above.  Then \[\eta_{dw}(\mathcal{Y}_{\gamma})\geq  \begin{cases}
\sqrt{\frac{1}{(3-2\gamma)^2}+1}, & \mbox{when}~~ 0<\gamma\leq\frac{1}{2} \\
\sqrt{(1-\gamma)^2+1}, & \mbox{when}~~ \frac{1}{2}\leq\gamma<1.
\end{cases}
\]  
\end{theorem}

\smallskip

 \noindent Similarly, for $\xi \in (0,1)$ we consider $\mathcal{X}_{\xi}= \{ (x,y,0) : x,y \in \mathbb{R} \} $ and define a norm $ {\Vert .\Vert}_\xi$ on $\mathcal{X}_{\xi}$ as \[{\Vert (x,y,0)\Vert}_\xi=\max \left\{\vert x\vert,\vert y\vert,\frac{\vert x\vert+\vert y\vert}{1+\xi}\right\} ~~\forall~~(x,y,0)\in \mathbb{R}^3. \] The unit ball $B_{\mathcal{X}_{\xi}}$ is an octagon with vertices $\pm(1,\xi,0),\pm(1,-\xi,0),\pm(\xi,1,0),$ $ \pm(\xi,-1,0)$. We also state the following theorem without proof.
\begin{theorem}\label{theorem:octagon}
Let $\mathcal{Y}_{\xi}$ be a $3$-dimensional polyhedral Banach space such that $B_{\mathcal{Y}_{\xi}}$ is a right prism whose base is an octagon $B_{\mathcal{X}_{\xi}}$ defined as above. Then 
\[\eta_{dw}(\mathcal{Y}_{\xi}) \geq 
\begin{cases}
    \sqrt{\frac{1}{(2+\xi)^2}+1},& \text{when } \xi>\frac{1-\xi}{1+\xi} \\
    \sqrt{\left(\frac{1+\xi}{3+\xi}\right)^2+1},& \text{when } \xi\leq\frac{1-\xi}{1+\xi}.
\end{cases}
\] 
\end{theorem}

\smallskip

%\begin{remark}
%We would like to notice that there exist Banach spaces $\mathcal{X}$ for which $\eta_{dw}(\mathcal{X})\neq \sqrt{n^2(\mathcal{X})+1}.$ As for example, if we consider $\mathcal{H}$ is a complex Hilbert space then we see that $\eta_{dw}(\mathcal{H})=1$ and $n(\mathcal{H})=\frac{1}{2}$, (see Remark \ref{remneq} $(2)$), i.e., $\eta_{dw}(\mathcal{H})\neq \sqrt{n^2(\mathcal{H})+1}.$
%\end{remark}

\noindent We note that all the relevant results in Section \ref{sec3} remain unchanged in case of computing the modified Davis-Wielandt index of the corresponding spaces.\\

\noindent Computation of the exact Davis-Wielandt index of a given normed linear space seems to be a non-trivial problem in the general most case. We end the article with the following open question in this regard, that we could not answer:\\

\noindent Open question: Let $\mathcal{X}$ be a normed linear space. Find separate complete characterizations of the following properties:

$(i)$ $\eta_{dw}(\mathcal{X})=1,$

$(ii)$ $\eta_{dw}(\mathcal{X})= \sqrt{n^2(\mathcal{X})+1}.$

\bibliographystyle{amsplain}

\end{document}